\documentclass{article}
\usepackage[english]{babel}
\usepackage{amsmath,amssymb,amsthm,amscd}

\newcommand{\R}{\mathbb{R}}

\newcommand{\N}{\mathbb{N}}
\theoremstyle{definition} \newtheorem{df}{Definition}
\theoremstyle{plain}      \newtheorem{thm}[df]{Theorem}
\theoremstyle{plain}      \newtheorem{prop}[df]{Proposition}
\theoremstyle{plain}      
\theoremstyle{plain}      \newtheorem{lem}[df]{Lemma}

\DeclareMathOperator{\tr}{tr}

\title{Some properties of the nematic radial hedgehog in Landau-de Gennes' theory}
\author{Xavier Lamy \thanks{Universit\'e de Lyon,  CNRS UMR 5208, Universit\'e Lyon 1, Institut Camille Jordan, 43 blvd. du 11 novembre 1918, F-69622 Villeurbanne cedex, France. Email address: xlamy$@$math.univ-lyon1.fr}}

\date{May 30, 2012}

\begin{document}

\maketitle

\begin{abstract} In the Landau-de Gennes theoretical framework of a $Q$-tensor description of nematic liquid crystals, we consider a radial hedgehog defect with strong anchoring conditions in a ball $B\subset \R^ 3$. We show that the scalar order parameter is monotonic, and we prove uniqueness of the minimizing hedgehog below the spinodal temperature $T^*$.
\end{abstract}

\section{Introduction and notations}

In the theory of Landau and de Gennes, a nematic liquid crystal is described by its tensor order parameter, the so-called $Q$-tensor \cite{degennes}, which is a traceless symmetric matrix : denoting by $S_3(\R)$ the set of symmetric $3\times 3$ matrices with real entries, it holds
\[Q\in S_3(\R),\quad \tr Q = 0.\]

A null $Q$-tensor corresponds to an isotropic liquid, a $Q$-tensor with two equal (non zero) eigenvalues correspond to a uniaxial nematic, and a $Q$-tensor with three distinct eigenvalues corresponds to a biaxial nematic \cite{degennes}.

In the case of a uniaxial nematic, the $Q$-tensor can be written
\begin{equation}\label{uniaxial}Q=s(n\otimes n -\frac{1}{3}\mathbf{1}),\end{equation}
where $n$ is a unit vector and $s\in\R$. The unit vector $n$ is called the director of the nematic : it corresponds to the preferred direction, along which the molecules tend to align. The scalar $s$ is called the scalar order parameter : it quantifies the degree of alignment of the molecules along this director \cite{degennes}. The value $s=0$ corresponds to the isotropic state.

We are going to consider a nematic crystal liquid filling a cavity $\Omega\subset\R^3$. The system is spatially inhomogeneous : the tensor order parameter $Q(x)$ depends on $x\in\Omega$, and at equilibrium it should minimize the energy
\[ E(Q) = \int_{\Omega} f_b(Q) + f_e(Q) \; dx, \]
where
\begin{equation}\label{fb} f_b(Q)= \frac{1}{2} a \tr(Q^2) - \frac{1}{3} b\tr (Q^3) +\frac{1}{4} c[\tr (Q^2)]^2,\end{equation}
and
\[ f_e(Q) = \frac{L}{2}|\nabla Q|^2.\]
The bulk energy density $f_b$ describes the nematic-isotropic transition in a spatially homogeneous system, while $f_e$ accounts for elastic deformations of the nematic. Here we work with a particularly simple form of the elastic energy density $f_e$, using the so-called one-constant approximation \cite{degennes}. For the bulk energy density $f_b$, terms of higher order could be taken into account, but the expansion up to fourth order is usually considered -- because it is the simplest one that allows to describe the transition \cite{degennes}. The constants $a,b,c$ and $L$ are material- and temperature-dependent, and in practice only $a$ depends significantly on the temperature. This dependence is usually assumed to be affine \cite{degennes}:
\[a=a_0(T-T^*),\]
 where $T^*$ is the so-called spinodal temperature.

It can be shown \cite[Proposition 1]{majumdar2010} that, for spatially homogeneous systems, the $Q$ tensors minimizing the bulk energy density $f_b$ are necessarily uniaxial (or isotropic), i.e. of the form (\ref{uniaxial}).

For uniaxial $Q$-tensors of the form (\ref{uniaxial}), $f_b$ depends only on the scalar order parameter $s$ :
\[f_b(Q)=f_b(s)=\frac{a}{3}s^2 - \frac{2b}{27}s^3 + \frac{c}{9}s^4.\]

One can distinguish between different temperature regimes, with three critical temperatures $T^*<T_c<T^+$. 
\begin{itemize}
\item In the temperature regime where $T<T^*$, i.e. $a<0$, the isotropic state is unstable : $f_b(s)$ has no local minimum at $s=0$, and has a global minimum at some value $s_+>0$. 
\item For $T^*<T<T_c$, $f_b(s)$ still has a global minimum at some $s_+>0$, but also has a local minimum at $s=0$.
\item For $T_c<T<T^+$, $f_b(s)$ still has a local minimum at some $s_+>0$, but the global minimum is at $s=0$.
\item Above $T^+$ the only critical point of $f_b(s)$ is at $s=0$.
\end{itemize}

We always consider the regime where there is a positive global minimizer of $f_b(s)$ (i.e. $T<T_c$) : the nematic state is stable (but the isotropic state also). For the uniqueness result (Theorem~\ref{unique thm}) we will further restrict ourselves to the case where there is no local minimizer at $s=0$ (i.e. $T<T^*$) : the isotropic state is not stable.

Here we study a nematic droplet : the container is
\[\Omega = B_R = \left\lbrace x\in\R^3 : |x|< R\right\rbrace.\]

A proper rescaling \cite{mkaddemgartland99} allows to work with dimensionless variables and to get rid of the material-dependent constants : we can take the energy to be
\begin{equation}\label{energy}E(Q) =\int_{B_R} \Big( \frac{1}{2}|\nabla Q|^2+f_b(Q)\Big)\: dx, \end{equation}
with the rescaled bulk energy density
\[f_b(Q)=\frac{t}{2}\tr(Q^2)-\sqrt{6}\tr(Q^3)+\frac{1}{2} \tr(Q^2)^2 + C(t).\]

Here $t=27 ac/b^2=27a_0(T-T^*)c/b^2$ is a dimensionless temperature and $C(t)$ is a constant chosen so that 
$\min f_b = 0.$ 

This non-dimensionalization of the problem is described in \ref{nondim} below.

In terms of the reduced temperature $t$, the critical temperatures are $0<1<9/8$ \cite{mkaddemgartland99}, and the different temperature regimes mentioned above are:
\begin{itemize}
\item For $t<0$, the isotropic state is not a local minimizer of $f_b$. (This corresponds to $T<T^*$.)
\item For $0<t<1$, the isotropic state is a local minimizer of $f_b$ but the nematic state is the global minimizer. (This corresponds to $T^*<T<T_c$.)
\item For $1<t<9/8$ the global minimizer of $f_b$ is the isotropic state. However, the nematic state is a local minimizer. (This corresponds to $T_c<T<T^+$.)
\item For $t>9/8$ the nematic state is not defined anymore, since the only local mimimizer is isotropic. (This corresponds to $T>T^+$.)
\end{itemize}
Hence in the following we assume $t<1$, and for the uniqueness result (Theorem~\ref{unique thm}) $t<0$.

We restrict ourselves to the case of strong radial anchoring, which means that we have the Dirichlet condition
\begin{equation}\label{Dirichlet} Q(x) = \sqrt{\frac{3}{2}}h_+ \Big(\frac{x}{|x|}\otimes \frac{x}{|x|} - \frac{1}{3}\mathbf{1} \Big)\quad\quad\text{for }|x|=R,\end{equation}
where $h_+$ is the scalar order parameter minimizing $f_b$: $f_b(Q)=0$ if and only if $Q$ is uniaxial with scalar order parameter $\sqrt{3/2}h_+$. Its value is given by \cite{mkaddemgartland99}
\[h_+ = \frac{3+\sqrt{9-8t}}{4}.\]

 The admissible space for $Q$ is then
\[W_Q=\left\lbrace Q\in W^{1,2}(B_R, S_3(\R)), \tr Q=0,\: Q \text{ satisfies }(\ref{Dirichlet})\right\rbrace.\] 

In general the minimizing $Q$-tensor may not exhibit complete spherical symmetry: symmetry breaking configurations have been observed experimentally and theoretically (see \cite{kraljvirga01} and references therein). However, here we restrict ourselves to the study of spherically symmetric critical points of $E$. 

To this end we define an action of the special orthogonal group $SO(3)$ on the space of admissible $Q$-tensor configurations $W_Q$, in the following way: 
\[( U\cdot Q)(x) =  U Q({}^tUx) {}^tU,\quad U\in SO(3).\]

The spherically symmetric configurations, i.e. those $Q\in W_Q$ satisfying
\[U\cdot Q = Q \quad \quad \forall U\in SO(3),\]
are precisely the maps of the form
\begin{equation}Q_h(x) = \sqrt{\frac{3}{2}}h(|x|)\Big(\frac{x}{|x|}\otimes\frac{x}{|x|}-\frac{1}{3}\mathbf{1}\Big),\end{equation}\label{Qh}
where the scalar order parameter $h$ belongs to the admissible space
\[W_h = \left\lbrace h\in L^2(0,R;dr), h'\in L^2(0,R; r^2dr), h(R)=h_+\right\rbrace.\]
An equilibrium configuration described by such $Q$ is  called radial hedgehog.

Hence we consider, for $h\in W_h$, the energy functional
\begin{equation}\label{I}I(h)=\frac{1}{4\pi}E(Q_h)=\int_0^R r^2\Big(\frac{1}{2} |h'|^2+ \frac{3|h|^2}{r^2} + g(h)\Big) dr,\end{equation}
where
\[g(h)=f_b(Q_h)=\frac{t}{2} h^2-h^3+\frac{1}{2}h^4+C(t).\]
The main goal of this article is to study minimizers of this functional, and in particular their monotonicity and uniqueness.

The problem of studying droplets of nematic liquid crystals is of great physical relevance. It plays an important role in some electro-optic applications, like polymer dispersed liquid crystals (PDLC) devices (see the review article \cite{lopezleon11} and references therein). Moreover, defects in general, and the radial hedgehog defect in particular, usually exhibit universal features \cite{kraljvirga01}, which makes it important to understand their structure better. The monotonicity of the scalar order parameter, which is proved in this article, is physically intuitive and had been observed numerically \cite{gartlandmkaddem00}, but no general mathematical proof of it was available. The uniqueness of this minimizing profile constitutes the main result of the present article.

The spherically symmetric critical points of $E$ have already been studied in a recent paper by A.~Majumdar \cite{majumdar2011}. A.~Majumdar also addresses the questions of uniqueness and monotonicity of the scalar order parameter $h$, in the limiting cases of low temperature ($t\to -\infty$) and big radius ($R\to\infty$), using techniques based on analogies with Ginzburg-Landau vortices. In the present paper these questions are studied without any great restriction on temperature and radius,  using elementary properties of the energy functional $I$, and the differential equation its critical points satisfy.

It should be mentioned here that there exist other continuum theories for nematic liquid crystals. The two most important are the Oseen-Frank model and the Ericksen model, which can both be viewed as simplifications of the Landau-de Gennes theory. In particular they only account for uniaxial configurations. For more information on the mathematical theory of liquid crystal, see the review article \cite{linliu01} and the references therein.

The paper is organized as follows. In Section~\ref{firstprop}, we recall a few basic properties of minimizers of $I$. Among them is a singular ordinary differential equation (\ref{ODE}) satisfied by the scalar order parameter. In Section~\ref{cauchy} we will see that, even though this equation (\ref{ODE}) is singular, it leads to a well-defined Cauchy problem. In Section~\ref{g} we show that $h$ is monotonic. And in Section~\ref{unique} we prove our main result, Theorem~\ref{unique thm} asserting that $I$ has a unique minimizer (in the temperature regime $T<T^*$).

\section{First properties}\label{firstprop}

The following proposition summarizes the first basic properties of the radial hedgehog. It is very similar to \cite[Proposition 2.1.]{majumdar2011}, but we give a proof here for completeness. It strongly relies on the symmetries of the problem.

\begin{prop}
\begin{itemize}
\item[(i)]
There exists a minimizer $h\in W_h$ of $I(h)$. It satisfies the singular ordinary differential equation
\begin{equation}\label{ODE}
h''(r)+\frac{2}{r}h'(r)-\frac{6h(r)}{r^2} = g'(h(r)),\quad r\in(0,R).
\end{equation}
\item[(ii)]
The corresponding $Q$-tensor $Q_h\in W_Q$ (as defined by (\ref{Qh}) above) is a critical point of the non-restricted energy functional $E$, $h$ is analytical on $[0,R)$, and it holds
\[h(0)=h'(0)=0.\]
\item[(iii)] For every $r\in [0,R]$,
\[0\leq h(r) \leq h_+.\]
\end{itemize}
\end{prop}
\begin{proof}
Let $(h_n)$ be a minimizing sequence for $I$ :
\[h_n\in W_h,\quad I(h_n)\longrightarrow \min_{W_h} I.\]
Define $Q_n=Q_{h_n}\in W_Q$. Since
\[\int_{B_R} |\nabla Q_n|^2 \: dx\leq 2 E(Q_n)= 8\pi I(h_n),\]
and $Q_n$ has fixed boundary conditions independent of $n$, the sequence $(Q_n)$ is bounded in the Hilbert space $\mathcal{W}:=W^{1,2}(B_R,S_3(\R))$, so, up to taking a subsequence, it
converges weakly to some $Q\in \mathcal{W}$.

Since the integrand defining the energy functional $E$ is convex in $\nabla Q$, $E$ is weakly lower semi-continuous (see e.g. \cite[Theorem 1.3]{dacorogna}), so that
\[ E(Q)\leq \liminf_{n\to\infty} E(Q_n) = 4\pi \liminf_{n\to\infty} I(h_n) = 4\pi\inf_{W_h} I.\]

The set of spherically symmetric $Q$-tensors is weakly closed in $\mathcal{W}$ because it is a closed subspace of $\mathcal{W}$. Therefore the weak limit $Q$ has to be a spherically symmetric $Q$-tensor: there exists $h\in W_h$ such that $Q=Q_h$. It holds 
\[ I(h)=\frac{1}{4\pi} E(Q)\leq \inf_{W_h} I,\]
hence $h$ is a minimizer of $I(h)$.

Classically, the fact that $h$ is a critical point of $I$ provides us with equation (\ref{ODE}): considering, for small $t$ a perturbation $h+t\varphi$, where $\varphi\in C_c^{\infty}(0,R)$, the vanishing of the first variation reads
\[\int_0^R r^2\big(h'\varphi' + \frac{6}{r^2}h\varphi+ g'(h)\varphi\big)\:dr=0,\]
which becomes, after an integration by part,
\[\int_0^R \big(-r^2 h'' -2rh' + 6r^2 h + r^2 g(h) \big)\varphi\: dr =0,\]
therefore (\ref{ODE}) holds in the sense of distributions in $(0,R)$, and hence it holds almost everywhere. This proves $(i)$.

To prove $(ii)$, we use the Principle of Symmetric Criticality \cite{palais}. For clarity, we state this principle in the form we will use here : if a group $G$ acts linearly and isometrically on a Hilbert space $H$, and $\Sigma$ is the set of points fixed by this action, then for any $G$-invariant functional $f\in C^1(H,\R)$, a critical point $x\in\Sigma$ of $f_{|\Sigma}$ is a critical point of the non-restricted functional $f$.

The action of $SO(3)$ defined above is linear and isometric. Our energy functional $E$ is invariant under this action, and it is $C^1$ because $f_b$ is $C^1$, $|f_b(Q)|\leq c|Q|^6$ and $|f_b'(Q)|\leq c|Q|^5$. 

Hence the Principle of Symmetric Criticality applies, and $Q_h$ is a critical point of the energy functional $E$ : it satisfies the Euler-Lagrange equations
\begin{equation}\label{euler}\Delta Q_{ij}= t\:Q_{ij} - 3\sqrt{6}\: \Big(Q_{ik}Q_{kj} - \frac{1}{3}\delta_{ij} \tr(Q^2)\Big) + 2 \: Q_{ij}\tr(Q^2).\end{equation}

Elliptic regularity theory \cite[Theorem 6.7.6]{morrey} tells us  that $Q_h$ is analytical on $B_R$, so that $h$ is analytical on $[0,R)$, since it holds, for any unit vector $u_0\in\R^3$,
\begin{equation}\label{analytic}h(r)=\sqrt{\frac{3}{2}}u_0 \!\cdot\! Q(ru_0) u_0. \end{equation}

As a consequence, it must hold
\[h(0)=0,\]
since 
\[|\nabla Q|^2(x) = h'(|x|)^2+\frac{6h(|x|)^2}{|x|^2}\]
is continuous. And it must further hold 
\[h'(0)=0,\]
since formula (\ref{analytic})
extends $h$ to a smooth and even function. It concludes the proof of $(ii)$. 

The analyticity of $h$ away from the origin was also clear from the ordinary differential equation (\ref{ODE}) it satisfies (and also at $r=R$ with this argument), but regularity at the origin was not obvious, since functions in the admissible space $W_h$ may even not have a finite limit at the origin.

As to $(iii)$, the upper bound follows from a maximum principle \cite[Proposition 3]{majumdarzarnescu10} for equation (\ref{euler}) satisfied by $Q_h$. The lower bound follows from 
\[0\leq I(|h|)-I(h) = \int_0^R r^2(h^3-|h|^3)\: dr = -2 \int_{h<0} r^2|h|^3 \:dr,\]
since $I(h)$ is minimal. Hence, the set $\{h<0\}$ is empty, and $h\geq 0$ as claimed.
\end{proof}

\section{The ODE satisfied by $h$}\label{cauchy}

We have seen that solutions of (\ref{ODE}) which lie in $W_h$ must verify $h(0)=h'(0)=0$. Their first degree of freedom is $h''(0)$. As a matter of fact, one can rewrite equation (\ref{ODE}) as
\[
\frac{d}{dr}\Big[ r^{-4}\frac{d}{dr}\big(r^ 3 h(r)\big)\Big] = r^{-1} g'(h(r)),\quad r\in(0,R),
\]
and use the following lemma to show the existence of a unique solution of ($\ref{ODE}$) with fixed second derivative at the origin.

Even though the ordinary differential equation (\ref{ODE}) is singular at the origin, Lemma~\ref{cauchypb} shows that this singularity is, in some sense, just a pseudo-singularity: there is still a well-defined Cauchy problem at the origin, and the proof is similar to the non singular case (Picard's Existence Theorem).

\begin{lem}\label{cauchypb}
Let $\alpha,\beta\in\R$, $\alpha <1$. Define $\gamma := 1 -\alpha -\beta$ and assume $\gamma\geq 0$. Let $F:\R\to\R$ be locally Lipschitz with $F(0)=0$. Let $a\in\R$. There exists  some $R_a>0$ and a function $h_a:(0,R_a)\to\R$, unique solution of
\begin{eqnarray}
\frac{d}{dr}\Big[ r^{\alpha}\frac{d}{dr}\big[r^{\beta} h(r)\big]\Big] & = & r^{\alpha+\beta} F(h(r)),\quad \quad \forall r\in(0,R_a) \label{hF}\\
\text{and}\quad\frac{h(r)}{r^\gamma} & \rightarrow & a , \quad\quad\quad\quad\quad\quad\quad\text{as}\quad r\to 0.\label{hFa}
\end{eqnarray}

\end{lem}
\begin{proof} It suffices to show existence and uniqueness in a neighbourhood of the origin.

Let us first remark that the problem (\ref{hF})-(\ref{hFa}) is equivalent, near zero, to the integral equation
\begin{equation}\label{hint}h(r)=ar^{\gamma}+r^{-\beta}\int_0^r \rho^{-\alpha}\int_0^{\rho} t^{\alpha+\beta} F(h(t))\:dtd\rho\quad\quad\forall r\in (0,\epsilon),\end{equation}
for a continuous $h$ verifying 
\begin{equation}\label{h*}|h(r)|\leq C r^{\gamma}\quad\forall r\in (0,\epsilon).
\end{equation}

Indeed, if $h$ satisfies the boundary condition (\ref{hFa}), then there exists $\epsilon$ and $C$ such that $h$ verifies (\ref{h*}). If, furthermore, $h$ satisfies (\ref{hF}), then an integration of (\ref{hF}) yields
\[ \frac{d}{dr}\Big[r^{\beta}h\Big] = A r^{-\alpha} + r^{-\alpha}\int_0^r t^{\alpha+\beta}F(h)\:dt.\]
A second integration (which is possible since $\alpha <1 $) leads to:
\[ h(r)= B r^{-\beta} + \frac{A}{1-\alpha} r^{\gamma} + r^{-\beta}\int_0^r \rho^{-\alpha}\int_0^{\rho} t^{\alpha+\beta} F(h)\: dt d\rho.\]
The boundary condition (\ref{hFa}) imposes $B=0$ (because $-\beta < \gamma$ as $\alpha<1$) and $A=a(1-\alpha)$.
 
 Conversely, if a continuous $h$ verifies (\ref{hint}) and (\ref{h*}), then derivating twice yields (\ref{hF}). Moreover the second term in the right-hand side of (\ref{hint}) is negligible compared to $ar^{\gamma}$ so that the boundary condition (\ref{hFa}) holds.

Let $K$ be the Lipschitz constant of $F$ on $[-|a|-1,|a|+1]$. Let \[\delta=\sqrt{\frac{2(3-\alpha)}{(1+|a|)K}}\] and $\epsilon =\min (\delta,1)$. We define a sequence of functions $h_n : (0,\epsilon]\to\R$ in the following way : 
\begin{equation}\label{hn}
h_{n+1}(r)=ar^{\gamma}+r^{-\beta}\int_0^r \rho^{-\alpha}\int_0^{\rho} t^{\alpha+\beta} F(h_n(t))\:dtd\rho,\end{equation}
and $h_0(r)=ar^{\gamma}$. The sequence $(h_n)$ is well-defined if for all $n$, 
\begin{equation}\label{wd}|h_n(r)|\leq (1+|a|)r^{\gamma}\quad\quad\forall r \in (0,\epsilon ].\end{equation}
Indeed, (\ref{wd}) implies the convergence of the integral in (\ref{hn}). We prove (\ref{wd}) by induction. Suppose (\ref{wd}) is true for some $n$. Then, making use of
\[|F(y)|=|F(y)-F(0)|\leq K|y|\quad\quad\text{for }|y|\leq 1+|a|\]
we obtain, using (\ref{hn}) and the induction assumption, we find that
\[|h_{n+1}(r)-ar^{\gamma}|\leq K(1+|a|) r^{-\beta}\int_0^r \rho^{-\alpha}\int_0^{\rho} t^{\alpha+\beta} t^{\gamma}\:dtd\rho = \frac{r^2}{\delta^2}r^{\gamma}.\]
This completes the induction since $r^2/\delta^2\leq 1$ for $r\in (0,\epsilon]$.
 
We next show that, for every $r\in (0,\epsilon ]$ and every $n\in\N$, we have
\begin{equation}\label{conv}|h_{n+1}(r)-h_n(r)|\leq (1+|a|) \frac{C^{n}r^{2n+\gamma}}{n!},\end{equation}
where $C=K/(6-2\alpha)$.

This follows by induction via the following calculation :
\begin{eqnarray*}
|h_{n+1}(r)-h_n(r)| & \leq & r^{-\beta}\int_0^r \rho^{-\alpha}\int_0^{\rho} t^{\alpha+\beta} \big| F(h_n(t))-F(h_{n-1}(t))\big| \:dtd\rho\\
& \leq & K r^{-\beta}\int_0^r \rho^{-\alpha}\int_0^{\rho} t^{\alpha+\beta} \big|h_n(t)-h_{n-1}(t)\big|\:dtd\rho\\
& \leq & (1+|a|) K \frac{C^{n-1}}{(n-1)!}r^{-\beta}\int_0^r \rho^{-\alpha}\int_0^{\rho} t^{\alpha+\beta} t^{2n - 2 + \gamma}\:dtd\rho\\
& = & (1+|a|) K \frac{C^{n-1}}{(n-1)!} \frac{r^{\gamma + 2n}}{2n(2n-\alpha +1)}\\
& \leq & (1+|a|) \frac{C^n r^{\gamma+2n}}{n!}.
\end{eqnarray*}

Inequality (\ref{conv}) ensures that the sequence $(h_n)$ converges uniformly on $(0,\epsilon ]$, to some continuous function $h$. The domination $|h_n(r)|\leq (1+|a|)r^\gamma $ allows to take the limit for $n\to\infty$ in (\ref{hn}), so that $h$ verifies the integral equation (\ref{hint}): $h$ is a solution to problem (\ref{hF})-(\ref{hFa}).

To show uniqueness, assume $\tilde{h}$ is a solution of (\ref{hF})-(\ref{hFa}). Then for some  $\epsilon$ it holds  \[|\tilde{h}(r)-ar^{\gamma}|\leq r^{\gamma}\quad\quad\forall r\in (0,\epsilon),\] and (\ref{hint}). An induction similar to the one performed above allows to show
\[|\tilde{h}(r)-h_n(r)|\leq (1+|a|) \frac{C^n r^{2n+\gamma} }{n!}\quad\quad\forall r\in (0,\epsilon),\]
so that $h=\tilde{h}$ and the problem has a unique solution.
\end{proof}

Taking $\alpha=-4$, $\beta=3$ and thus $\gamma=2$, Lemma~\ref{cauchypb} directly applies to our equation, yielding the result claimed above.

\section{Monotonicity of h}\label{g}

Here we keep assuming that we are in the temperature regime $T<T_c$ (i.e. $t<1$).

\begin{prop}
If $h\in W_h$ is a local minimizer of $I$ and satisfies $0\leq h \leq h_+$, then $h$ is increasing on $(0,R)$. In addition, we have
\[h'>0\quad in\quad (0,R].\]
\end{prop}
\begin{proof} First we will show that
\[ h' \geq 0 \quad\mathrm{in}\quad(0,R).\]

Indeed, let
\[W_{h,0}=\left\{\psi\in L^2(0,R;dr), \psi'\in L^2(0,R;r^2dr), \psi(R)=0 \right\}.\]
For $\psi\in W_{h,0}$, we have, as $s\to 0$,
\begin{equation}\label{var} I(h+s\psi)=I(h)+sI'(h)\cdot \psi + \frac{s^2}{2} I''(h)\cdot (\psi,\psi) + o(s^2),\end{equation}
with
\[I'(h)\cdot\psi = \int_0^R r^2 (h'\psi'+6 \frac{h\psi}{r^2} +g'(h)\psi)\: dr\]
and 
\[I''(h)\cdot (\psi,\psi) = \int_0^R r^2 ( \psi'^2 + 6 \frac{\psi^2}{r^2} + g''(h)\psi^2)\:dr.\]
Validity of (\ref{var}) follows easily from the fact that $g$ is polynomial and $h$ is bounded. Since $h$ is a local minimizer, we have $I'(h)\cdot\psi = 0$ (this is the Euler Lagrange equation) and $I''(h)\cdot (\psi,\psi) \geq 0$ for every $\psi\in W_{h,0}$.

Our analysis of the Cauchy problem in Section~\ref{cauchy} implies that $h'>0$ on $(0,\delta)$ for some positive $\delta$. Let us assume that $h$ is not monotonic in $(0,R)$: there exist $0<R_1<R_2< R$ such that $h'=0$ at $R_1$ and $R_2$ and $h'<0$ on $(R_1,R_2)$.

Let $\psi = h'\mathbf{1}_{(R_1,R_2)}\in W_{h,0}$. We are going to show that $I''(h)\cdot (\psi,\psi) < 0$, which will contradict the fact that $h$ is a local minimizer.

Differentiating once the equation (\ref{ODE}) satisfied by $h$, one gets
\begin{equation}\label{h'}
 h^{(3)} + \frac{2}{r} h'' - \frac{8}{r^2} h' + \frac{12}{r^3} h = g''(h)h'.
\end{equation}
Multiplying (\ref{h'}) by $r^2 h'$ and integrating from $R_1$ to $R_2$, we obtain, after an integration by parts (using $h'(R_1)=h'(R_2)=0$),

\[\int_{R_1}^{R_2}\Big( -r^2 (h'')^2 - 8 (h')^2 + \frac{12}{r} h h'\Big)  \;dr = \int_{R_1}^{R_2} r^ 2 g''(h) (h')^2 dr.\]
This can be rewritten as
\[I''(h)\cdot (\psi,\psi) = \int_{R_1}^{R_2} \frac{12}{r} h h' \: dr -2 \int_{R_1}^{R_2} (h')^2 \:dr.\]
We obtain the desired contradiction since $h\geq 0$ and $h'<0$ in $(R_1,R_2)$.

Eventually we can actually show that
\[h'>0\quad \mathrm{in}\quad (0,R].\]

Assume first that for some $r_0 \in (0,R)$ we have $h'(r_0)=0$ then since $h'\geq 0$ it must hold $h''(r_0)=0$. Therefore, using equation (\ref{h'}) we get
\[h^{(3)}(r_0)=-\frac{12}{r_0^3}h(r_0)<0,\]
contradicting $h'\geq 0$.

It remains to show $h'(R)>0$. Assume $h'(R)=0$. Using equation (\ref{ODE}) satisfied by $h$  together with the fact that
\[g'(h(R))=g'(h_+)=0,\]
we have
\[h''(R) = 6 \frac{h(R)}{R^2}>0,\]
again contradicting $h'\geq 0$.
\end{proof}

\section{Uniqueness of h}\label{unique}

In this section we assume $t<0$, so that the function $g$ is decreasing in $[0,h_+]$.

\begin{thm}\label{unique thm}
In the temperature regime $T<T^*$, the energy functional $I$ admits a unique global minimizer $h\in W_h$.
\end{thm}
\begin{proof}The Euler Lagrange equations (\ref{euler}), may be rewritten as
\begin{equation}\label{euler2}
\Delta Q_{ij} = \frac{\partial f_b}{\partial Q_{ij}}(Q) + \sqrt{6}\: \delta_{ij} \tr (Q^2).
\end{equation}
Solutions of (\ref{euler2}) satisfy the following Pohozaev identity \cite[Lemma~2]{majumdarzarnescu10}: 
\[
\frac{1}{R}\int_{\partial B_R}\!\!\!\! |x\cdot\nabla Q|^2\: d\sigma(x)  - 
\frac{R}{2}\int_{\partial B_R} \!\!\!\!|\nabla Q|^2 \: d\sigma(x) + 
\frac{1}{2}\int_{B_R}\!\!\!\! |\nabla Q|^2\: dx = - 3 \int_{B_R}\!\!\!\! f_b(Q)\: dx,
\]
where $d\sigma$ stands for the usual measure on the sphere $\partial B_R$.

Using the definition of the energy $E$, this identity is equivalent to

\begin{equation}\label{poh}
E(Q)+2 \int_{B_R}\!\!\!\! f_b(Q)\: dx = \frac{R}{2}\int_{\partial B_R}\!\!\!\! |\nabla Q|^2 \: d\sigma(x) - \frac{1}{R}\int_{\partial B_R}\!\!\!\! |x\cdot\nabla Q|^2 \: d\sigma(x).
\end{equation}

Now let us assume there exist $h_1 \neq h_2\in W_h$ satisfying 
\[I(h_1)=I(h_2)=\min_{W_h} I.\] Applying (\ref{poh}) to $Q_{h_i}$ for $i=1,2$ and subtracting the two equations, we find
\begin{equation}\label{uni}
8 \pi \int_0^R r^2\big( g(h_1) - g(h_2)\big) dr = 2 \pi \big( h_2'(R)^2 - h_1'(R)^2\big).
\end{equation}

The two functions $h_1$ and $h_2$ are distinct solutions of the Euler Lagrange equation (\ref{ODE}) corresponding to the minimization of $I$.
Hence Lemma~\ref{cauchypb} ensures that they cannot coincide on a neighbourhood of zero : we have for instance $h_2>h_1$ on $(0,\epsilon)$. Actually they can never meet before $R$ : let us assume the existence of a $r_0\in (0,R)$ such that $h_1(r_0)=h_2(r_0)$. Then we obtain a new minimizer $\tilde{h}$ of $I$,
\[\tilde{h}= h_2 \mathbf{1}_{(0,r_0)} + h_1 \mathbf{1}_{(r_0,R)}.\]

Let us prove the fact that $\tilde{h}$ is a minimizer : firstly it lies in the admissible space, and, denoting by $e[h]$ the energy density ($I(h)=\int_0^R e[h] dr$), we have
\[I(h_2)\leq I(h_1 \mathbf{1}_{(0,r_0)} + h_2 \mathbf{1}_{(r_0,R)}) = \int_0^{r_0} e[h_1] \: dr + \int_{r_0}^R e[h_2] \: dr,\]
since $h_2$ is a minimizer and $h_1 \mathbf{1}_{(0,r_0)} + h_2 \mathbf{1}_{(r_0,R)}$ lies in the admissible space. Therefore it holds
\[\int_0^{r_0} e[h_2] \: dr \leq \int_0^{r_0} e[h_1] \: dr,\]
and adding $\int_{r_0}^R e[h_1] dr$ on both sides of this inequality exactly yields $I(\tilde{h})\leq I(h_1)$, hence proving the affirmation.

In particular, since $\tilde{h}$ is a minimizer it must be analytic. But at $r_0$ all its right derivatives are equal to those of $h_1$, thus $\tilde{h}=h_1$ on a neighbourhood of $r_0$, so that $h_1=h_2$ which contradicts our primary assumption. Therefore it holds \[h_1<h_2\quad \mathrm{on} \quad(0,R).\] In particular it implies, together with $h_1(R)=h_2(R)$, that
\[h_1'(R)\geq h_2'(R), \]
so the right-hand side of (\ref{uni}) is non positive.

On the other hand, since the energy density $g$ is decreasing on $[0,h_+]$, and $h_1<h_2$,  the left-hand side of (\ref{uni}) is positive : we obtain a contradiction.
\end{proof}

A different but related question is whether limiting maps obtained in the limit $R\to\infty$ satisfy the ordinary differential equation (\ref{ODE}) with the limiting boundary condition $h(\infty)=h_+$. Uniqueness of solutions to this boundary value problem, without minimality assumptions, is established in \cite{ignatnguyen}. For a discussion about uniqueness in the limit $t\to\infty$, see also \cite[Section~3]{henaomajumdar} and the references therein.

\section{Acknowledgements}

The author thanks P.Mironescu for many instructive discussions and very helpful advice. He also thanks A.Zarnescu for useful discussions. He learned from A.Zarnescu that in a forthcoming paper \cite{ignatnguyen}, R.Ignat, L.Nguyen, V.Slastikov and A.Zarnescu prove a uniqueness result for global solutions of (\ref{ODE}) with $h(\infty)=h_+$. Eventually he thanks the anonymous referee for his/her careful review and several suggestions to improve the article (in particular bringing \cite{henaomajumdar} to his attention).

\appendix

\section{Non-dimensionalization}\label{nondim}

We define, as in \cite{mkaddemgartland99}, the typical length $\xi$ and the reduced temperature $t$, in terms of which one can express a dimensionless version of the problem :
\[\xi=\sqrt{\frac{27cL}{b^2}},\quad t=\frac{27 ac}{b^ 2}.\]

For $Q\in W_Q$ define $\tilde{Q}$ by
\[\tilde{Q}(\tilde{x})=q_0 Q(\xi\tilde{x})\quad\text{for}\;|\tilde{x}|< \tilde{R}:=\frac{R}{\xi}\]
where
\[q_0=\sqrt{\frac{27c^2}{2b^2}}.\]

Making the change of variable $x=\xi\tilde{x}$ in the integral (\ref{energy}) defining the energy then gives
\begin{eqnarray*}
E(Q ) & = & \int_{B_{\tilde{R}}} \Big( \frac{L}{2q_0^2\xi^2}|\nabla \tilde{Q}|^2 + f_b \big(\frac{1}{q_0}\tilde{Q} \big) \Big)\:\xi^3 d\tilde{x} \\
& = & \int_{B_{\tilde{R}}} \Big(\frac{L\xi}{2q_0^2}|\nabla\tilde{Q}|^2 + \frac{a\xi^3}{2q_0^2}\tr(\tilde{Q}^2)-\frac{b\xi^3}{3q_0^3}\tr(\tilde{Q}^3)+\frac{c\xi^3}{4q_0^4}\tr(\tilde{Q}^2)^2 \Big)\: d\tilde{x}.
\end{eqnarray*}
Direct computation shows that
\[ \frac{L\xi}{q_0}=\frac{1}{t}\frac{a\xi^3}{q_0^2}
=\frac{1}{\sqrt{6}}\frac{b\xi^3}{3q_0^3} = 2 \frac{c\xi^3}{4q_0^3}=\sqrt{\frac{4b^2L^3}{27c^3}}.\]
Therefore it holds
\[\sqrt{\frac{27c^3}{4b^2L^3}}E(Q) = \tilde{E}(\tilde{Q}),\]
where the rescaled energy $\tilde{E}$ is defined by
\[\tilde{E}(\tilde{Q})=\int_{B_{\tilde{R}}} \Big( \frac{1}{2}|\nabla \tilde{Q}|^2 + \frac{t}{2}\tr(\tilde{Q}^2) - \sqrt{6} \tr(\tilde{Q^3}) + \frac{1}{2}\tr(\tilde{Q}^2)^2\Big)\: d\tilde{x}.\]

\bibliographystyle{plain}
\bibliography{uniqueness}
\end{document}